\documentclass[12pt]{amsart}
\usepackage{amsmath,amsfonts,textcomp,longtable,pstricks}
\usepackage[dvips]{graphicx}

\newtheorem{thm}{Theorem }[section]

\newtheorem{lemma}[thm]{Lemma }

\newtheorem{prop}[thm]{Proposition }

\newtheorem{corollary}[thm]{Corollary }

\theoremstyle{definition}

\numberwithin{equation}{section}

\def\RR{{\mathbb R}}

\def\QQ{{\mathbb Q}}
\def\ZZ{{\mathbb Z}}

\def\FF{{\mathbb F}}
\def\fq{{\mathbb{F}_q}}
\def\kk{{\bar{k}}}

\def \bra#1\ket {\mathop{\vphantom{#1}\left<\smash{#1}\right>}\nolimits}
\def\sep{{\mathrm sep}}

\DeclareMathOperator{\coker}{coker}

\DeclareMathOperator{\Gal}{Gal} 
 \DeclareMathOperator{\rk}{rk}
 \DeclareMathOperator{\diag}{diag}

 \DeclareMathOperator{\Np}{Np}
\DeclareMathOperator{\Hp}{Hp_\ell}

\def\OO{\mathcal{O}}

\def\plim{\mathop{{\lim\limits_{\longleftarrow}}}\nolimits}

\renewcommand \phi {\varphi}

\begin{document}
\author{Sergey Rybakov}
\thanks{Supported in part by the research project DST-1211005, RFBR grants no. 12-01-31280, 12-01-92697-IND-a, 14-01-93108 and by a subsidy granted to the HSE by the Government of the Russian Federation for the implementation of the Global Competitiveness Program}
\address{Institute for information transmission problems of the Russian Academy of Sciences}
\address{Poncelet laboratory (UMI 2615 of CNRS and Independent University of
Moscow)}
\address{ AG Laboratory, HSE, 7 Vavilova str., Moscow, Russia, 117312 }

\email{rybakov@mccme.ru, rybakov.sergey@gmail.com}%
\dedicatory{To M. A. Tsfasman and S. G. Vl\v{a}du\c{t} on the occasion of their 60th birthdays}%
\title[On groups of points on abelian varieties]
{On classification of groups of points\\ on abelian varieties over finite fields}
\date{}


\begin{abstract}
In this paper we improve our previous results on classification of groups of points on abelian varieties over finite fields.
The classification is given in terms of the Weil polynomial of abelian varieties in a given $k$-isogeny class over a finite field $k$.
\end{abstract}

\maketitle
\section{Introduction}
Let $X$ be an algebraic variety over a finite field $k=\fq$  of
characteristic $p$. The set $X(k)$ of points defined over $k$ is an important 
invariant of $X$. If $X$ is an abelian variety, then $X(k)$ is a finite abelian group. 
Previously~\cite{Ry13} we classified groups of points on abelian surfaces over finite fields.
It is important that for the classification in question one first divide the set of abelian surfaces into 
isogeny classes, and then classify groups of points inside a given isogeny class. In this paper we follow 
the same strategy. By the Tate-Honda theorem, an isogeny class of abelian varieties over 
$k$ corresponds to characteristic polynomial of Frobenius action on the $\ell$-th Tate module of 
any variety from the class (the Weil polynomial). These polynomials are known for abelian varieties of low 
dimensions. In this paper we state some partial results concerning higher dimensions. In particular, we prove a conjecture stated in~\cite{Ry1}. We also clarify the proof of the classification theorem for groups of points on abelian surfaces. 

\section{Preliminary results on abelian varieties and notation}
Let $\kk$ be an algebraic closure of $k$. For a given 
prime number $\ell$ and a natural number $m$ denote by $A_m$ the
kernel of multiplication by $\ell^m$ in $A(\kk)$. The $\ell$-th Tate module of $A$ is defined as $T_\ell(A) =
\plim A_m$. The Frobenius endomorphism $F$ of $A$ acts on the Tate module by a
semisimple linear transformation, which we also denote by $F$.
Suppose $\ell\ne p$. Then $T_\ell(A)$ is a free $\ZZ_\ell$-module of rank $2\dim A$.
We define {\it the Weil polynomial of $A$} as the characteristic polynomial
$$
f_A(t) = \det(t-F|T_\ell(A)).
$$
It is a monic polynomial of degree $2\dim A$ with rational integer coefficients independent of
the choice of prime $\ell$. 
Tate proved that abelian varieties $A$ and $B$ are isogenous if and only if $f_A(t)=f_B(t)$~\cite{Ta66}. 

In what follows we would like to treat the case $\ell=p$ as well. The module $T_p(A)$ sometimes is called the physical Tate module. 
The polynomial $g(t)=\det(t-F|T_p(A))$ divides $f_A(t)$ in $\ZZ_p[t]$. Moreover, $\deg g\leq \dim A$, and $$f_A(t)\equiv g(t)t^{2\dim A-\deg g}\mod p.$$ 
In fact, main results of this paper are trivial for $T_p$, because its rank is too small for abelian varieties of low dimension. However we think it is more natural to state the results for all primes. 

Recall some notation and results from~\cite[Section 2]{Ry} .
Let $\Delta$ be the set of roots of $f\bmod\ell$, and let $\Lambda\subset\Delta$ be the image of a section of the natural map $\Delta\to\Delta/\Gal(\overline{\FF}_\ell/\FF_\ell)$. 
Then $f=\prod_{\lambda\in\Lambda}f_\lambda$, where $f_\lambda\in\ZZ_\ell[t]$ is the monic polynomial such that the set of roots of $f_\lambda\bmod\ell$ is $\Gal(\overline{\FF}_\ell/\FF_\ell)\lambda$.
We have a corresponding decomposition of the Tate module
$$T_\ell(A)\cong\oplus_{\lambda\in\Lambda} T_\lambda(A),$$ where $F$ acts on $T_\lambda(A)$ with the characteristic polynomial $f_\lambda$.
In this paper we are interested in the summand corresponding to $\lambda=1$ which we denote by $T_\ell^{(1)}(A)$. We also write $$T_\ell^{(2)}(A)=\oplus_{\lambda\neq 1} T_\lambda(A).$$
We denote by $f_\ell\in \ZZ_\ell[t]$ the polynomial $f_1$ corresponding to $\lambda=1$. This is the divisor of $f_A$ of maximal degree such that $f_\ell(t-1)\equiv t^{\deg f_\ell}\bmod\ell$.
Note that this results hold for $\ell=p$ without any changes.

Let $f_\ell^{\sep}\in \ZZ_\ell[t]$ be the monic separable polynomial with the same set of roots as $f_\ell$, and let $$R=\ZZ_\ell[t]/f_\ell^{\sep}\ZZ_\ell[t].$$ Denote by $x\in R$ the
image of $t$ under the natural projection from $\ZZ_\ell[t]$. Thus $T_\ell^{(1)}(A)$ is an $R$-module such that $x$ acts as $F$. 
Without danger of confusion we define {\it a Tate module} as an $R$-module which is free of finite rank over $\ZZ_\ell$. The {\it rank} of a Tate module is its rank over $\ZZ_\ell$. 

For an abelian group $G$ we denote by $G_\ell$ the $\ell$-primary
component of $G$. The group $A(k)$ is the kernel of $1-F: A\to A$, and the $\ell$-component
$$A(k)_\ell\cong T_\ell(A)/(1-F)T_\ell(A)$$ (see~\cite[Proposition 3.1]{Ry1}). 
By definition, $T_\ell(A)/(1-F)T_\ell(A)\cong T_\ell^{(1)}(A)/(1-x)T_\ell^{(1)}(A)$. Informally, $T_\ell^{(1)}(A)$ is the minimal part of $T_\ell(A)$ which knows about the group of points. The proof of
the following proposition is essentially the proof of~\cite[Theorem~1.1]{Ry1}. 

\begin{prop}
\label{prop_on_Tate_module} 
Let $f$ be a Weil polynomial, and let $G$ be a finite abelian group of order $f(1)$.  
Suppose that for any prime number $\ell$ dividing $f(1)$ there exists a Tate module $T_\ell$ such that $f_\ell(t)=\det(t-x|T_\ell)$, and $G_\ell\cong T_\ell/(1-x)T_\ell$. 
Then there exists an abelian variety $A$ over $k$ such that $A(k)\cong G$. 
\end{prop}
\begin{proof}
Let $B$ be an abelian variety over $k$ with Weil polynomial $f$. By the discussion before the proposition, $T_\ell(B)\cong T_\ell^{(1)}(B)\oplus T_\ell^{(2)}(B)$. Since the Frobenius action on the vector space $V_\ell^{(1)}(B)=T_\ell^{(1)}(B)\otimes\QQ_\ell$ is semisimple, it is determined up to isomorphism by the polynomial $f_\ell$. Therefore there exists an isomorphism $V_\ell^{(1)}(B)\cong T_\ell\otimes_{\ZZ_\ell}\QQ_\ell$, and an inclusion $T_\ell\to V_\ell^{(1)}(B)$ such that the image of $T_\ell$ is contained in $T_\ell^{(1)}(B)$. Put $T'_\ell=T_\ell\oplus T_\ell^{(2)}(B)$. Clearly, we have an isomorphism of $F$-vector spaces $T'_\ell\otimes\QQ_\ell\to V_\ell(B)$. 
By~\cite[Lemma 2.1]{Ry1}, there exists an abelian variety $B_1$ and an $\ell$-isogeny $B_1\to B$ such that
$T_\ell(B_1)\cong T'_\ell$. By~\cite[Proposition 3.1]{Ry1}, we have
$G_\ell\cong B_1(k)_\ell$.

Let $\ell_1,\dots,\ell_s$ be the set of prime divisors of
$f(1)$. By the above there is a sequence of isogenies
$$A=B_s\xrightarrow{\phi_s} B_{s-1}\to\dots\xrightarrow{\phi_2}
B_1\xrightarrow{\phi_1} B$$ such that $\phi_i:B_i\to B_{i-1}$ is
an $\ell_i$-isogeny and $$B_i(k)_{\ell_i}\cong G_{\ell_i}.$$ Since $\phi_i$ is an
$\ell_i$-isogeny, $T_\ell(B_i)\cong T_\ell(B_{i-1})$ for any
$\ell\neq\ell_i$. Thus $A(k)\cong G$.
\end{proof}

Fix and a positive integer $r$. A {\it matrix factorization} (with respect to $f_\ell, f_\ell^{\sep}$, and $r$) is a pair $(X,Y)$ of $r\times r$ matrices with coefficients in $\ZZ_\ell[t]$ such that $\det X=f_\ell$ and $YX=f_\ell^{\sep}\cdot I_r$, where $I_r$ is the identity matrix. The following proposition shows the connection between Tate modules and matrix factorizations. 

\begin{prop}\cite{Ry13}\label{prop_mf}
Let $(X,Y)$ be a matrix factorization. Then \begin{equation}\label{eq_mf}
T=\coker(\ZZ_\ell[t]^r \xrightarrow{X} \ZZ_\ell[t]^r).
\end{equation}
is a Tate module, and the characteristic polynomial of the action of $x$ on $T$ is equal to $f_\ell$.

Conversely, let $T$ be a Tate module, which can be generated over $R$ by $r$ elements. Suppose that $\det(t-x|T)=f_\ell$, and $\Hp(T/(1-x)T,r)=(m_1,\dots,m_r)$. Then there exists a matrix
factorization $(X,Y)$ such that $T$ has the presentation
${\rm(\ref{eq_mf})}$, and $$X\equiv \diag(\ell^{m_1},\dots,\ell^{m_r})\mod (1-t)\ZZ_\ell[t].$$
\end{prop}

\section{Tate modules and Hodge polygons}
Denote by $\nu_\ell$ the $\ell$-adic valuation on $\QQ_\ell$.
Let $Q(t)=\sum_i Q_i t^i$ be a polynomial of degree $d$ over
$\QQ_\ell$, and let $Q(0)\neq 0$. Take the lower convex hull
of the points $(i,\nu_\ell(Q_i))$ for $0\leq i\leq d$ in $\RR^2$.
The boundary of this region without vertical lines is called {\it
the Newton polygon $\Np_\ell(Q)$ of $Q$}. 

Let $G=\oplus_{i=1}^r\ZZ/\ell^{m_i}\ZZ$ be an $\ell$--group of order $\ell^m$. In what follows we assume that the exponents $m_1,\dots,m_r$ are ordered in the following way: $0\leq m_1\leq m_2\leq\dots\leq m_r$.
Note that some of these numbers could be zero.
The {\it Hodge polygon $\Hp(G,r)$ of the group $G$} is the convex polygon with
vertices $(i,\sum_{j=1}^{r-i}m_j)$ for $0\leq i\leq r$. It has
$(0,m)$ and $(r,0)$ as its endpoints, and its slopes are
$-m_r,\dots, -m_1$. We write $\Hp(G,r)=(m_1,\dots,m_r).$ If $T$ is a Tate module, we denote $\Hp(T/(1-x)T,\rk T)$ by $\Hp(T)$.

\begin{lemma}\label{key_lemma}
Let $$0\to H\to G\to G'\to 0$$ be a short exact sequence of finite abelian $\ell$--groups. 
Let $\Hp(H,s)=(n_1,\dots,n_s)$, and $\Hp(G,r)=(m_1,\dots,m_r)$.
Assume that $G'$ is generated by $r-s$ elements. Then $m_i\leq n_i$ for $1\leq i\leq s$.

Dually, if $\Hp(G',r-s)=(n'_1,\dots,n'_{r-s})$, then $m_i\geq n'_i$ for $s+1\leq i\leq r$.
\end{lemma}
\begin{proof}
The group $G$ is isomorphic to $$\oplus_{i=1}^r \ZZ/\ell^{m_i}\ZZ.$$ 
Choose $v_1,\dots,v_r\in G$ such that $v_i$ generates the summand $\ZZ/\ell^{m_i}\ZZ$. We choose generators $u_1,\dots,u_s\in H$ in the same way. Let $H'$ be a subgroup of $H$ generated by $u_1,\dots,u_i$.
The group $G'$ is generated by $r-s$ elements; thus $G/H'$ is generated by $r-i$ elements, and there exists a linear combination $$v=\sum_{j=i}^r a_jv_j$$ such that $v\in H'$, and there exists $j$ such that $\nu_\ell(a_j)=0$. Now assume that $m_i> n_i$. It follows that $\ell^{n_i}v\neq 0$. On the other hand, $\ell^{n_i}u=0$ for any $u\in H'$. This contradiction concludes the proof.

The last assertion can be proved by the dual argument.
\end{proof}

\begin{thm}\label{main}
Let $T$ be a Tate module of rank $r$, and $T_1$ be a Tate submodule of rank $s$ such that $T_2=T/T_1$ is torsion-free. Suppose that $\Hp(T_1)=(n_1,\dots,n_s)$, and $\Hp(T)=(m_1,\dots,m_r)$.
Then $m_i\leq n_i$ for $1\leq i\leq s$. 

Dually, if $\Hp(T_2)=(n'_1,\dots,n'_{r-s})$, then $m_i\geq n'_i$ for $s+1\leq i\leq r$.
\end{thm}
\begin{proof} We prove that $m_i\leq n_i$ for $1\leq i\leq s$. The second assertion follows from the dual argument.
The transformation $E=1-x$ acts on the short exact sequence of Tate modules:
$$0\to T_1\to T \to T_2\to 0.$$
By the snake lemma we get a exact sequence of abelian groups:
$$0\to T_1/ET_1\to T/ET\to T_2/ET_2\to 0.$$
The Lemma~\ref{key_lemma} completes the proof.
\end{proof}

\begin{lemma}\label{lem1}
Let $L$ be a finite extension of $\QQ_\ell$ generated by an element $\alpha$ with $\ell$-adic valuation $\lambda\in\QQ$. Denote by $\OO$ the ring of integers of $L$. Then there exists an isomorphism of abelian groups: 
$$\OO/\alpha \OO\cong(\ZZ/\ell^{\lfloor\lambda\rfloor}\ZZ)^{[L:\QQ_\ell]-e}\oplus(\ZZ/\ell^{\lceil\lambda\rceil}\ZZ)^{e},$$ 
where $\lambda=\lfloor\lambda\rfloor+e/[L:\QQ_\ell]$.
\end{lemma}
The proof is an easy application of~\cite[Proposition I.7.23]{Lang}.

\begin{thm}\label{main_cor}
Let $T$ be a Tate module of rank $r$, and let $\Hp(T)=(m_1,\dots,m_r)$. Put $Q(t)=f(1-t)$.
Let $\alpha$ be a root of $Q$ with multiplicity $s$ and $\ell$-adic valuation $$\lambda=\lfloor\lambda\rfloor+\frac{e}{[L:\QQ_\ell]}\text{, and let } s'=\left\lceil\frac{([L:\QQ_\ell]-e)s}{[L:\QQ_\ell]}\right\rceil,$$ where $L=\QQ_\ell(\alpha)$.
Then $$m_{s'}\leq\lfloor\lambda\rfloor,~~m_s\leq\lceil\lambda\rceil,~~m_{r-s'+1}\geq\lceil\lambda\rceil\text{, and }m_{r-s+1}\geq\lfloor\lambda\rfloor.$$
\end{thm}
\begin{proof}
Denote by $\OO$ the ring of integers of $L$. Let $S=T\otimes_{\ZZ_\ell}\OO$. 
The Tate module $S$ contains the submodule $S_1=\ker(1-x-\alpha)$. Clearly, $$S_1/(1-x)S_1\cong (\OO/\alpha \OO)^s\text{, and }S/(1-x)S\cong(T/(1-x)T)\otimes_{\ZZ_\ell}\OO\cong(T/(1-x)T)^{[L:\QQ_\ell]}.$$ By Lemma~\ref{lem1}, the Hodge polygon of $(\OO/\alpha \OO)^s$ 
has $([L:\QQ_\ell]-e)s$ slopes $\lfloor\lambda\rfloor$ and $es$ slopes $\lceil\lambda\rceil$. Thus the slope number $[L:\QQ_\ell]s$ is equal to $\lceil\lambda\rceil$.
By Theorem~\ref{main}, $m_{s'}\leq\lfloor\lambda\rfloor$, and $m_s\leq\lceil\lambda\rceil$.
The rest can be proved dually.
\end{proof}

\begin{corollary}\label{main_cor2}
Let $A$ be an abelian variety over $k$, and let $\Hp(T_\ell^{(1)}(A))=(n_1,\dots,n_r)$.
Put $T=\wedge^e T_\ell(A)$, and $\Hp(T)=(m_1,\dots,m_{r'})$, where $r'=\binom{r}{e}$. Clearly, for any $i$ there exist $i_1< \dots < i_e$ such that $m_i=n_{i_1}+\dots +n_{i_e}$.
Let $Q(t)=\det(t-E|T)$, where $E$ is induced by $1-x$.
Let $\alpha$ be a root of $Q$ with multiplicity $s$ and $\ell$-adic valuation $\lambda\in\QQ$. 
Then $m_s\leq\lceil\lambda\rceil$, and $m_{r'-s+1}\geq\lfloor\lambda\rfloor$.
\end{corollary}

\section{Main results}
\begin{thm}\label{conj}
Let $f=P^rQ^s$ be a Weil polynomial, where $Q$ divides $P$, and $\deg P\leq 2$.
Suppose that $P$ is separable. Then $G$ is the group of points on some abelian variety with Weil polynomial $f$ if and only if $$G_\ell\cong \oplus_{j=1}^r G_\ell^{(j)}\oplus H,$$ where $G_\ell^{(j)}$ are $\ell$-primary abelian groups such that $\Np_\ell(P(1-t))$ lies on or above $\Hp(G_\ell^{(j)},2)$ for all $1\leq j\leq r$, and $H\cong(\ZZ_\ell/Q(1)\ZZ_\ell)^s$.
\end{thm}
\proof
If $\deg P_\ell\leq1$, the statement is trivial. We assume that $P_\ell=P$, and $\deg P=2$; therefore $f_\ell=f$. 
It is easy to construct a variety with a given group of points. By Proposition~\ref{prop_on_Tate_module}, we have to construct Tate modules $T_\ell$ over $R=\ZZ_\ell[t]/P(t)\ZZ_\ell[t]$ such that $T_\ell/(1-x) T_\ell\cong G_\ell$. By~\cite[Theorem 3.2]{Ry1}, there exist Tate modules $T_\ell^{(j)}$ such that $$T_\ell^{(j)}/(1-x) T_\ell^{(j)}\cong G_\ell^{(j)}$$ and $P$ is the characteristic polynomial of Frobenius action on $T_\ell^{(j)}$. Clearly, there exists a Tate module $T_H$ such that $T_H/(1-x) T_H\cong H$ and $Q^s$ is the characteristic polynomial of the Frobenius action on $T_H$. Put $T_\ell=\oplus_j T_\ell^{(j)}\oplus T_H$.

Let $m=\nu_\ell(P(1))$, and let $\lambda_1\leq\lambda_2$ be slopes of $\Np_\ell(P(1-t))$.
We prove by induction on $\deg f$ that if $T$ is an $R$-module and the characteristic polynomial of the action of $t$ is $f=P^rQ^s$, then $$T/(1-x)T\cong \oplus_{j=1}^r G_\ell^{(j)}\oplus(\ZZ_\ell/Q(1)\ZZ_\ell)^s,$$ where  $\Hp(G_\ell^{(j)},2)=(m_j,m_{2r+s-j})$ such that $m_j\leq\lambda_1$, and $m_j+m_{2r+s-j}=m$ for $1\leq j\leq r$. 

The statement is clear, if $\deg f\leq 2$. Assume $\deg f>2$. Let $\Hp(T)=(m_1,\dots,m_{2r+s})$. 
Then  $$T/(1-x) T\cong\oplus_{i=1}^{2r+s} \ZZ/\ell^{m_i}\ZZ.$$ 
Choose $v_1,\dots,v_{2r+s}\in T/(1-x) T$ such that $v_i$ generates the summand $\ZZ/\ell^{m_i}\ZZ$. Suppose that $w\in T$ lifts $v_{2r+s}$. Let $T_1\subset T$ be the saturated sublattice containing $w$ and $Fw$. 
Clearly, $T_1$ is $F$-invariant.
If the characteristic polynomial of the Frobenius action on $T_1$ is $Q$, then $T_1/(1-x) T_1\cong\ZZ_\ell/Q(1)\ZZ_\ell$. Thus $\Hp(T/T_1)=(m_1,\dots,m_{2r+s-1})$, and we use induction hypothesis.

Suppose that the characteristic polynomial of the Frobenius action on $T_1$ is $P$. Then $$T_1/(1-x) T_1\cong \ZZ/\ell^{m-m_{2r+s}}\ZZ\oplus\ZZ/\ell^{m_{2r+s}}\ZZ.$$ 
By Theorem~\ref{main}, 
\begin{equation}\label{eq1}
m_1\leq m-m_{2r+s}.
\end{equation}

Now we use the argument from the proof of~\cite[Theorem 3.4]{Ry13}.
By Proposition~\ref{prop_mf}, $T$ corresponds to a matrix factorization $(X,Y)$ such that $\det X=f$ and
$$X\equiv \diag(\ell^{m_1},\dots,\ell^{m_{2r+s}})\mod (1-t)\ZZ_\ell[t].$$
The matrix factorization $(Y,X)$ corresponds to a module $T'$ over
$R$, which is generated by $2r+s$ elements and the characteristic
polynomial of $x$ is equal to $$\det Y=P^{2r+s}/f.$$ Moreover,
$$Y\equiv \diag(\ell^{m-m_1},\dots,\ell^{m-m_{2r+s}})\mod (1-t)\ZZ_\ell[t],$$
i.e., $$\Hp(T')=(m-m_{2r+s},\dots,m-m_1).$$ 
As before, either $m_1=\nu_\ell(Q(1))$, and we are reduced to the case $\deg f=r+s-1$, or we have the inequality~(\ref{eq1}) for $T'$:
$$m-m_{2r+s}\leq m-(m-m_1)=m_1.$$ It follows that $m=m_1+m_{2r+s}$.

We proved that $$T_1/(1-x) T_1\cong \ZZ/\ell^{m_1}\ZZ\oplus\ZZ/\ell^{m_{2r+s}}\ZZ,$$ thus, by Theorem~\ref{main_cor}, $m_1\leq\lambda_1$.

As in the proof of Theorem~\ref{main}, we have an exact sequence of $R$-modules:
$$0\to T_1\to T \to T_2=T/T_1\to 0,$$ and the corresponding sequence of abelian groups: $$0\to T_1/(1-x)T_1\to T/(1-x)T\to T_2/(1-x)T_2\to 0.$$
Clearly, $T_2$ is a free $\ZZ_\ell$-module, and the characteristic polynomial of $x$ on $T_2$ is $P^{r-1}Q^s$.
We claim that $$\Hp(T_2)=(m_2,\dots,m_{2r-2+s}).$$ Indeed, $T/(1-x)T$ is generated by $v_1,\dots,v_{2r+s}$ and, by construction of $T_1$, the group $T_1/(1-x)T_1$ is generated by $v_{2r+s}$ and some $v$ such that $\ell^{m_1}v=0$. Moreover, since $T_2$ is generated by $2r+s-2$ elements, we may assume that $$v=\sum_{j=1}^{2r+s} a_jv_j,~a_j\in\ZZ$$ where $\ell^{m_1}a_jv_j=0$, and there exists $j$ such that $\nu_\ell(a_j)=0$. In particular, $\nu_\ell(a_j)\geq m_j-m_1$. Thus we can reorder $v_1,\dots,v_j$ and assume that $\nu_\ell(a_1)=0$.
Now we see that the images of $v_2,\dots,v_{2r-2+s}$ generate the subgroup 
$$\oplus_{j=2}^{2r-2+s} \ZZ/\ell^{m_j}\ZZ$$ of the same order as $T_2/(1-x)T_2$.
This proves the claim, and the induction hypothesis for $T_2$ proves the theorem.
\qed

\begin{thm}\label{conj2}
Let $f(t)=P(t)(t\pm\sqrt{q})^r$ be a Weil polynomial, where $P$ is separable of degree $2$, and $P(\mp\sqrt{q})\neq 0$. 
Let 
\begin{itemize}
\item $G$ be a finite abelian group of order $f(1)$;
\item $\Hp(G_\ell,r+2)=(m_1,\dots,m_{r+2})$;
\item $\lambda_1\leq\lambda_2$ be slopes of $\Np_\ell(P(1-t))$;
\item $\lambda_q=\nu_\ell(1\mp\sqrt{q})$.
\end{itemize}
Then $G$ is the group of points on some abelian variety with Weil polynomial $f$ if and only if 
\begin{enumerate}
\item $\Np_\ell(f(1-t))$ lies on or above $\Hp(G_\ell,r+2)$;
\item $m_r\leq \lambda_q$;
\item $m_{3}\geq \lambda_q$;
\item $m_1+m_{r+1}\leq \lambda_q+\lambda_1$.
\end{enumerate}
\end{thm}
\proof
Let $G=A(k)$ for some $A$ with Weil polynomial $f$. Assertion $(1)$ follows from~\cite[Theorem 1.1]{Ry1}, and
inequalities $(2)$ and $(3)$ follow from Theorem~\ref{main_cor} applied to $$T=T_\ell^{(1)}(A),\text{ }s=r,~\alpha=1\pm\sqrt{q},\text{ and } \lambda=\lambda_q.$$ In particular, $m_3=\dots=m_r=\lambda_q$.
To prove $(4)$ apply Corollary~\ref{main_cor2} to  $$s=r,\text{ }\alpha=(1\mp\sqrt{q})^{r-1}(1-\alpha_1)\text{ and } \lambda=(r-1)\lambda_q+\lambda_1,$$ where $\alpha_1$ is a root of $P$ such that $\nu_\ell(1-\alpha_1)=\lambda_1$. We get that $$m_1+m_{r+1}+(r-2)\lambda_q=m_1+m_3+\dots+m_{r+1}\leq (r-1)\lambda_q+\lceil\lambda_1\rceil.$$ Let us prove that if $\lambda_1$ is not integral,
then $$m_1+m_{r+1}<\lambda_q+\lceil\lambda_1\rceil.$$ In this case $\lambda_1=\lambda_2=\nu_\ell(P(1))/2$. If $m_1+m_{r+1}=\lambda_q+\lceil\lambda_1\rceil$, then $m_2+m_{r+2}=\lambda_q+\lfloor\lambda_2\rfloor$. But from $m_1+m_{r+1}\leq m_2+m_{r+2}$ it follows that $\lceil\lambda_1\rceil\leq\lfloor\lambda_2\rfloor$.

We prove the theorem in other direction. 
We construct Tate modules $T_\ell$ such that $f(t)=\det(t-x|T_\ell)$, and 
$G_\ell\cong T_\ell/(1-x)T_\ell$ and use Proposition~\ref{prop_on_Tate_module}. 
If $f\neq f_\ell$ we can apply the construction of the Tate module from the proof of Theorem~\ref{conj}.
 If $\ell=p$, then $\deg P_\ell\leq 1$, and $f\neq f_\ell$. So assume that $f=f_\ell$, and $\ell\neq p$.
 Let $T_1$ be a Tate module such that $x$ acts with characteristic polynomial $P$.
Put $T_2=\ZZ_\ell$ such that $x$ acts as multiplication by $\mp\sqrt{q}$, and put $$V=(T_1\oplus T_2^2)\otimes\QQ_\ell.$$ 
In~\cite{Ry13} we construct a Tate module $T\subset V$ such that $\Hp(T)=(m_1,m_2,m_{r+1},m_{r+2})$. Put $T_\ell=T\oplus T_2^{r-2}$.
 \qed

\section{Abelian surfaces and threefolds}
By Theorems~\cite[Theorem 1.1]{Ry1},~\ref{conj} and~\ref{conj2} we get a more conceptual proof of~\cite[Theorem 2.4]{Ry13}.

\begin{thm}\label{2dim}
Let $A$ be an abelian surface over a finite field with Weil polynomial $f_A$. 
Let $G$ be an abelian group of order $f_A(1)$.
Then $G$ is the group of points on some variety in the isogeny class
of $A$ if and only if for any prime number $\ell$ one of the following conditions holds.
\begin{enumerate}
\item Suppose $f_A$ has no multiple roots, then
$\Np_\ell(f_A(1-t))$ lies on or above $\Hp(G_\ell, 4)$.
\item Suppose $f_A=P_A^2$, and $P_A$ has no
multiple roots, then $G_\ell\cong G_\ell^{(1)}\oplus G_\ell^{(2)}$, where $G_\ell^{(1)}$ and
$G_\ell^{(2)}$ are $\ell$-primary abelian groups with one or two generators
such that $\Np_\ell(P_A(1-t))$ lies on or above $\Hp(G_\ell^{(1)},2)$
and $\Hp(G_\ell^{(2)},2)$.
\item Suppose $f_A=P(t)(t\pm\sqrt{q})^2$, and $\sqrt{q}\in\ZZ$, where $P(t)$ is separable, and $P(\mp\sqrt{q})\neq 0$. Let
\begin{itemize}
\item $\Hp(G_\ell,4)=(m_1,m_2,m_3,m_4)$;
\item $\lambda_1\leq\lambda_2$ be slopes of $\Np_\ell(P(1-t))$;
\item $\lambda_q=\nu_\ell(1\pm\sqrt{q})$;
\end{itemize}
Then
\begin{enumerate}
\item $\Np_\ell(f(1-t))$ lies on or above $\Hp(G_\ell,4)$;
\item $m_2\leq \lambda_q$;
\item $m_3\geq \lambda_q$;
\item $m_1+m_3\leq \lambda_q+\lambda_1$.
\end{enumerate}
\item Finally, if $f_A(t)=(t\pm\sqrt{q})^4$, and $\sqrt{q}\in\ZZ$, then
$G\cong(\ZZ/(1\pm\sqrt{q})\ZZ)^4$.
\end{enumerate}
\end{thm}

Finally, we give a partial classification of groups of points on abelian threefolds. Note, that in the theorem below we do not know the classification in case $(2)$ for $r>3$,
 and in cases $(3)$ and $(4)$ for $r>2$.

\begin{thm}\label{3dim}
Let $A$ be an abelian threefold over a finite field with Weil polynomial $f_A$. Then the following conditions hold for $G=A(k)$:
\begin{enumerate}
\item Suppose $f_A$ is separable. Then $\Np_\ell(f_A(1-t))$ lies on or above $\Hp(G_\ell,6)$.
\item Suppose $f_A=P^2$, and $P$ is separable. Let $m=\nu_\ell(P(1))$, and let $G_\ell$ be generated by $r$ elements. Then $r\geq 2$. If  $r=2$, then $G_\ell\cong(\ZZ/m\ZZ)^2$, and if $r=3$, then $\Hp(G_\ell,3)=(m_1,m_2,m_3)$, where $\Np_\ell(P(1-t))$ lies on or above the Hodge polygon $(m-m_3,m-m_2,m-m_1)$.
\item Suppose $f_A=P^2Q$, where $\deg P=\deg Q=2$, and the polynomial $PQ$ is separable. Let $m=\nu_\ell(P(1)Q(1))$. Then $G_\ell$ is generated by $r\geq 2$ elements, and if $r=2$, then $\Hp(G_\ell,2)=(m_1,m_2)$, where $\Np_\ell(Q(1-t))$ lies on or above the Hodge polygon $(m-m_2,m-m_1)$.
\item 
Suppose $f_A(t)=P(t)(t\pm\sqrt{q})^2$, and $\sqrt{q}\in\ZZ$, where $P(t)$ is separable, and $P(\mp\sqrt{q})\neq 0$.
Let $m=\nu_\ell(P(1)(1\pm\sqrt{q}))$.
Then $G_\ell$ is generated by $r\geq 2$ elements, and if $r=2$, then $\Hp(G_\ell,2)=(m_1,m_2)$, where $\Np_\ell(P(1-t))$ lies on or above the Hodge polygon $(0,0,m-m_2,m-m_1)$.
\item 
Suppose $f_A(t)=P(t)(t\pm\sqrt{q})^4$, and $\sqrt{q}\in\ZZ$, where $P(t)$ is separable, and $P(\mp\sqrt{q})\neq 0$.
Let
\begin{itemize}
\item $\Hp(G_\ell,6)=(m_1,\dots,m_6)$;
\item $\lambda_1\leq\lambda_2$ be slopes of $\Np_\ell(P(1-t))$;
\item $\lambda_q=\nu_\ell(1\pm\sqrt{q})$.
\end{itemize}
Then 
\begin{enumerate}
\item $\Np_\ell(f(1-t))$ lies on or above $\Hp(G_\ell,6)$;
\item $m_2\leq \lambda_q$;
\item $m_3=m_4=\lambda_q$;
\item $m_5\geq \lambda_q$;
\item $m_1+m_5\leq \lambda_q+\lambda_1$.
\end{enumerate}
\item Finally, suppose $f_A=P^rQ^s$ for some $r$ and $s$, where $\deg P\leq 2$, the polynomial $P$ is separable, and $Q$ divides $P$. Then $G_\ell\cong \oplus_{j=1}^r G_\ell^{(j)}\oplus H$, where $G_\ell^{(j)}$ are $\ell$-primary abelian groups such that $\Np_\ell(P(1-t))$ lies on or above $\Hp(G_\ell^{(j)},2)$ for all $1\leq j\leq r$, and $H\cong(\ZZ_\ell/Q(1)\ZZ_\ell)^s$.
\end{enumerate}
Conversely, let $G$ be an abelian group of order $f_A(1)$ such that one of the conditions above holds.
Then $G$ is the group of points on some variety in the isogeny class of $A$.
\end{thm}
\begin{proof}
The cases $(1)$, $(5)$ and $(6)$ follow from Theorems~\cite[Theorem 1.1]{Ry1},~\ref{conj} and~\ref{conj2}. We prove $(2)$ first. Without loss of generality assume that $P=P_\ell$. 
Recall that $P(x)=0$, and $\nu_\ell(P(1))=m$. It follows that the group $$G_\ell\cong T_\ell(A)/(1-x)T_\ell(A)$$ is annihilated by $\ell^m$. Thus $r\geq 2$, and if $r=2$, then $G_\ell\cong(\ZZ/m\ZZ)^2$. Assume $r=3$. 
By Proposition~\ref{prop_on_Tate_module}, we have to find out when there exists a module $T$ over $R=\ZZ_\ell[t]/P(t)\ZZ_\ell[t]$ with $r=3$ generators such that $\Hp(T/(1-x)T,3)=(m_1,m_2,m_3)$.
By Proposition~\ref{prop_mf}, there exists a matrix factorization $(X,Y)$ such that $\det X=f_A$ and
$$X\equiv \diag(\ell^{m_1},\ell^{m_2},\ell^{m_3})\mod (1-t)\ZZ_\ell[t].$$
The matrix factorization $(Y,X)$ corresponds to a module $T'$ over
$R$, which is generated by $3$ elements, and the characteristic
polynomial of $x$ is equal to $$\det Y=P^3/f_A=P,$$ and
$$Y\equiv \diag(\ell^{m-m_1},\ell^{m-m_2},\ell^{m-m_3})\mod (1-t)\ZZ_\ell[t].$$
Such a module exists if and only if $\Np_\ell(P(1-t))$ lies on or above $\Hp(T'/(1-x)T',3)=(m-m_3,m-m_2,m-m_1)$.
We proved that there exists a module $T$ such that $\Hp(T/(1-x)T,3)=(m_1,m_2,m_3)$ if and only if $\Np_\ell(P(1-t))$ lies on or above $(m-m_3,m-m_2,m-m_1)$. 

In cases $(3)$ and $(4)$ the group $G_\ell$ is annihilated by $\ell^m$. Thus it is not cyclic, and $r\geq 2$. 
We now consider the case $(3)$. Let $R=\ZZ_\ell[t]/P(t)Q(t)\ZZ_\ell[t]$. Assume we have an $R$-module $T$ with $r=2$ generators such that $\Hp(T/(1-x)T,2)=(m_1,m_2)$.
As before, there exists a matrix factorization $(X,Y)$ such that $\det X=f_A$ and
$$X\equiv \diag(\ell^{m_1},\ell^{m_2})\mod (1-t)\ZZ_\ell[t].$$
The matrix factorization $(Y,X)$ corresponds to a module $T'$ over
$R$, which is generated by $2$ elements, and the characteristic
polynomial of $x$ is equal to $$\det Y=(PQ)^2/f_A=Q,$$ and
$$Y\equiv \diag(\ell^{m-m_1},\ell^{m-m_2})\mod (1-t)\ZZ_\ell[t].$$
As above, it follows that there exists a module $T$ such that $\Hp(T/(1-x)T,2)=(m_1,m_2)$ if and only if $\Np_\ell(Q(1-t))$ lies on or above $\Hp(T'/(1-x)T',3)=(m-m_2,m-m_1)$. 

Case $(4)$. Let $R=\ZZ_\ell[t]/P(t)(t\pm\sqrt{q})\ZZ_\ell[t]$, and let $T$ be an $R$-module with $r=2$ generators such that $\Hp(T/(1-x)T,2)=(m_1,m_2)$.
For the corresponding matrix factorization $(X,Y)$ the matrix factorization $(Y,X)$ gives a module $T'$ over
$R$, which is generated by $2$ elements, and the characteristic
polynomial of $x$ is equal to $$\det Y=(P(t\pm\sqrt{q}))^2/f_A=P,$$ and
$$Y\equiv \diag(\ell^{m-m_1},\ell^{m-m_2})\mod (1-t)\ZZ_\ell[t].$$
The rank of $T'$ is $4$, and such a module exists if and only if $\Np_\ell(P(1-t))$ lies on or above $\Hp(T'/(1-x)T',4)=(0,0,m-m_2,m-m_1)$.
Therefore there exists a module $T$ such that $\Hp(T/(1-x)T,2)=(m_1,m_2)$ if and only if $\Np_\ell(P(1-t))$ lies on or above $\Hp(T'/(1-x)T',3)=(m-m_2,m-m_1)$.  
\end{proof}

\end{document}